\documentclass[11pt, reqno]{article}
\usepackage{amsmath}  
\usepackage{amssymb}
\usepackage{amsfonts}
\usepackage{amsthm}

\usepackage[dvipsnames]{xcolor}

\usepackage{mathtools}

\usepackage[thinc]{esdiff}

\usepackage{epstopdf}
\usepackage[caption=false]{subfig}

\usepackage[numbers,sort&compress]{natbib}
\bibpunct[, ]{[}{]}{,}{n}{,}{,}
\makeatletter
\def\NAT@def@citea{\def\@citea{\NAT@separator}}
\makeatother

\theoremstyle{plain}

 \newtheorem{thm}{Theorem}[section]
 \newtheorem{cor}[thm]{Corollary}
 \newtheorem{lem}[thm]{Lemma}
 
 \theoremstyle{definition}
 \newtheorem{defn}[thm]{Definition}
 \theoremstyle{remark}

 \numberwithin{equation}{section}

\usepackage{enumerate}

\usepackage{tikz}

\usetikzlibrary{shapes.geometric}

\tikzstyle{c}=[circl ( 2cm), draw=none, fill=white, text=black]
\tikzstyle{place}=[circle, draw=blue!50,fill=blue!20,thick, inner sep=0pt, minimum size=1.5cm]
\tikzstyle{place1}=[ellipse, draw=blue!50,fill=blue!20,thick, inner sep=0pt, minimum size=0.5cm]                   

\tikzstyle{arrow1}=[thick,<->,>=stealth]
\tikzstyle{arrow2}=[thick,<->,>=stealth]

\usetikzlibrary{arrows}
\tikzstyle{arrowmid}[0.5]=[decoration= {markings, mark=at position #1 with {\arrow{>}}}, postaction={decorate}]

\tikzset{ kb1/.style={postaction={decorate, decoration={markings,mark=at position .5 with {\arrow{>};}}}},
kb2/.style={postaction={decorate, decoration={markings,mark=at position .2 with {\arrow{>};}}}},   }

\usetikzlibrary{arrows,decorations.markings}

\usetikzlibrary{angles}

\usetikzlibrary{calc}
\usepackage{relsize}

\tikzset{fontscale/.style = {font=\relsize{#1}}}

\usepackage{caption}

\usepackage{authblk}

\usepackage[colorlinks=true, all colors=blue]{hyperref}

\usepackage{hyperref}

\makeatletter
\renewcommand*{\@fnsymbol}[1]{\@alph{#1}}
\makeatother

\title{Differentiability of the largest Lyapunov exponent for planar open billiards}

\author {Amal Al Dowais 
 \thanks{Department of Mathematics and Statistics, School of Physics, Mathematics and Computing, University of Western Australia, Perth, WA 6009, Australia  \newline {\textit{\indent ~~ Email address:} \href{mailto: amal.aldowais@research.uwa.edu.au}{amal.aldowais@research.uwa.edu.au}}} $^{  , }$\thanks {Department of Mathematics, College of Science and Arts, Najran University, Najran, Saudi Arabia \newline{\textit{\indent ~~ Email address:} \href{mailto: amalduas@nu.edu.sa}{amalduas@nu.edu.sa}} }
  }

\date{}

\begin{document}










\maketitle

\begin{abstract}
In this paper, we estimate the largest Lyapunov exponent for open billiards in the plane. We show that the largest Lyapunov exponent is differentiable with respect to a billiard deformation.
\end{abstract}

{\small \bf {keywords.}} {\small {Open billiards; Lyapunov exponents; Non-wandering set; Billiard deformation}}

{\small \bf{Mathematics Subject Classification (2010).}} {\small {37D50, 37B10, 37D20, 34D08}}

\section{Introduction}

The stability and instability of a dynamical system can be studied by means of Lyapunov exponents. A dynamical system is considered chaotic if it has a positive Lyapunov exponent. Examples of chaotic systems are the dispersing billiards or so-called Sinai billiards (see \cite{Si1}, \cite{Si2}). Billiards are dynamical systems in which a particle moves with constant speed and hits the billiard’s wall (boundary of the billiard’s domain) according to the law of geometrical optics,\textquotedblleft the angle of incidence equals the angle of reflection”. Open billiards are a particular case of billiards in unbounded domains. The domain is the exterior of finitely many strictly convex compact obstacles satisfying the no-eclipse condition {\bf(H)} of Ikawa \cite{I}: the convex hull of any two obstacles does not intersect with another obstacle; in other words, there does not exist a straight line that intersects more than two obstacles. 
It follows from Sinai \cite{Si1}, \cite{Si2} (see also \cite {PS}) that the non-wandering set of the open billiard map is hyperbolic (i.e. there exist positive and negative Lyapunov exponents). 
Many studies have investigated Lyapunov exponents for billiards (see \cite{Wo}, \cite{Ba},  \cite{Ch2}, \cite{Mar1}, \cite{Mar2}). In this paper, we estimate the largest Lyapunov exponent for open billiard in $\mathbb{R}^2$. We demonstrate that the Lyapunov exponent depends continuously on a parameter $\alpha$ related to a deformation of the billiard as defined in \cite{W2}. Moreover, we prove that the Lyapunov exponent is differentiable with respect to the deformation parameter $\alpha$.

\bigskip

Here we state the main results:

\bigskip

\noindent In the following theorems, we denote the billiard deformation by $K(\alpha)$ where $\alpha \in [0, b]$. See Section \ref{defo} for the precise definition.

\bigskip

\begin{thm} {\normalfont {{\bf(Continuity)}}} 
Let $K(\alpha)$ be a $\mathcal{C}^{4, 2}$ billiard deformation in $\mathbb{R}^2$. Let $\lambda_1(\alpha)$ be the largest Lyapunov exponent for $K(\alpha)$. Then the largest Lyapunov exponent is continuous as a function of $\alpha$.

\label{con.}

\end{thm}

\bigskip

\begin{thm} {\normalfont {{\bf(Differentiability)}}}
Let $K(\alpha)$ be a $\mathcal{C}^{5,3}$ billiard deformation in $\mathbb{R}^2$. Let $\lambda_1(\alpha)$ be the largest Lyapunov exponent for $K(\alpha)$. Then $\lambda_1 (\alpha)$ is $\mathcal{C}^1$ with respect to $\alpha$.

\label{diff.}

\end{thm}

\bigskip

\noindent There are many works studying continuity properties of Lyapunov exponents (see e.g. \cite{V}, \cite{DK}). However, to our knowledge, in all these continuity is established generically,  i.e. with respect to “most” (typical) values of the parameters/perturbations involved. In the case of the open billiard considered in the present paper we establish continuity and even differentiability for all values of the parameter that appears in the perturbation, which is a truly remarkable property of this physical system.


\section{Preliminaries}

This section provides some preliminary concepts for open billiards, billiard flow, 
symbolic coding and stable/unstable manifolds. We also describe some notations related to curvatures, distances, and collision angles. In the last part of this section, we state the Oseledets multiplicative ergodic theorem and its consequence for open billiards.

\subsection{{{Open billiard}}}
\label{1}

Let $K_{i}$, where $i= 1,2,3, ...,z_0,$ be strictly convex compact domains with smooth boundaries $\partial K_i$ in $\mathbb{R}^2$. In this paper, we assume that $K= \bigcup_i K_i$ satisfies the following condition{\textbf {(H)}} of Ikawa \cite{I}: for any $ i \neq j \neq k$ the convex hull of $ K_{i} \cup K_{k}$ does not have any common points with $K_{j}$.  Let $\Omega $ be the exterior of $K$ (i.e., $ \Omega = \overline{\mathbb{R}^2 \backslash  K}$). 
Let $\Phi_t$, $t\in \mathbb{R}$, be the {\it billiard flow} such that for any particle $x= (q,v)$, where $q \in \Omega$ represents the position of $x$ and $v $ is the unit velocity of the particle $x$, then $\Phi_t(x) = (q_t, v_t) = (q + tv, v)$.  When the particle hits the boundary, then the velocity follows the collision law $v_{new}= v_{old} -2 <v_{old} , n> n$, where $n$ is the outwards unit normal vector to $\partial K$ at $q$, and $\phi $ the angle between $n= n(q)$ and  $v$.

 We denote the time of the $j$-th reflection of $x$ by $t_j(x) \in (-\infty, \infty)$ for $j \in \mathbb{Z}$. We say $t_{j}(q) = \infty$ ($t_{j}(q)  = -\infty$) if the forwards (backwards) trajectory of $x$ has less than $j$ reflections. We denote the {\it non-wandering} set of the flow $\Phi_t$ by $\Lambda = \{ x \in \widetilde{\Omega}, |t_{j}(x)| < \infty , \, \, \, for \, \, all \, \, j\in \mathbb{Z}\}$, where 
 $$\widetilde{\Omega} = \{(q,v) \in int \, \, \Omega \times \mathbb{S}^{1} \, \, or \,\, (q,v) \in  \partial \Omega \times \mathbb{S}^{1} : \langle n(q),v \rangle \geq 0 \}.$$ 

 Now let $ M = \{x=(q,v) \in \partial K \times \mathbb{S}^{1} : \langle n(q),v \rangle \geq 0 \}$ and let $ \pi : M \to \partial K$ be the canonical projection map defined by $\pi (q,v) = q$. Let $t_1(x)$ be the time of the first reflection of $x$ and let $M_{1} = \{x \in M : t_{1}(x) < \infty\}$. Define the {\it billiard ball map} $B: M_{1} \to M$ by $ B(x)=\Phi_ {t_{1}(x)} (x)$, (e.g. if $y=(p_0, w_0)$, where $p_0$ lies on $\partial K_i$ then $B(y)= B(p_0,w_0)= (p_1, w_1)$ where $p_1 = p_0 +t_1 w_0 \in \partial K_j$ and $w_1= w_0 - 2 < w_0, n > n$). The non-wandering set of the open billiard map is $M_0 = \{x \in M : |t_{j}(x)| < \infty\}$ which is a subset of $\Lambda$.  Finally, let $B : M_0 \to M_0 $ be the restriction of the open billiard map on the {\it non-wandering} set $M_{0}$. It is obvious that the non-wandering set is an invariant set. See \cite{Si1}, \cite{Si2}, \cite{Ch2}, \cite{ChM}, \cite{PS}, for general information about billiard dynamical systems. 


\bigskip

\subsection{{Symbolic coding for open billiards}}

\label{coding}

Each particular $x \in M_0$ can be coded by a {\it bi-infinite sequence} $$ \xi(x) =  (...,\xi_{-1},\xi_0,\xi_1,...) \in \{1,2,...,z_0\}^\mathbb{Z},$$ in which $ \xi_i \neq \xi_{i+1}$, for all $i \in \mathbb{Z}$, and $ \xi_j $ indicates the obstacle $K_{\xi_j}$ such that $\pi B^j(x) \in \partial K_{\xi_j}$. For example, if there are three obstacles $K_1,K_2$ and $K_3$ as above and a particular $q$ repeatedly hits $K_1,K_3, K_2, K_1,K_3, K_2$, then the bi-infinite sequence is $(...,1,3,2,1,3,2,...)$. Let $\Sigma $ be the {\it symbol space} which is defined as: $$\Sigma = \{ \xi = (...,\xi_{-1},\xi_0,\xi_1,...) \in \{1,2,...,z_0\}^\mathbb{Z} : \xi_i \neq \xi_{i+1} , \forall i \in \mathbb{Z} \}.$$ Define the {\it representation map} $ R: M_0 \to \Sigma$ by $R(x) = \xi(x)$. Let $\sigma: \Sigma \to  \Sigma $ be the {\it two-sided subshift} map 
defined by $\sigma(\xi_i) = \xi_{i+1}$. Given $\theta \in (0,1)$ define the metric $d_\theta $ on $\Sigma$ by: 

$$
d_\theta(\xi, \eta) =     \left\{ \begin{array}{lcl} 0 & \mbox {if} &  \xi_i = \eta_i  \ \   \mbox {for all} \ i \in \mathbb{Z}  \\
        \theta^n & \mbox{if } & n = max \{j \geq 0 : \xi_i = \eta_i  \ \  \mbox{for all} \ |i| < j \} 
                \end{array}\right.
$$

\noindent Then $\sigma$ is continuous with respect to $d_\theta$ (\cite{B}). It is also known that the representation map 
$R: M_0 \to \Sigma$ is a homeomorphism (see e.g. \cite{PS}). See \cite{I}, \cite{LM}, \cite{Mor}, \cite{PS}, \cite{St1},  for topics related to symbolic dynamics for open billiards.

\bigskip

\subsection{Lyapunov exponents}

Here we state a consequence of Oseledets Multiplicative Ergodic Theorem for billiards (see e.g. Ch. 2 in \cite{Po}, also see \cite{Os}, \cite{V}, \cite{KS}). 

\bigskip











\noindent For the open billiard map $B: M_0 \longrightarrow M_0$ in $\mathbb{R}^2$
we will use the coding $R: M_0 \longrightarrow \Sigma$ from Section \ref {coding}, which conjugates
$B$ with the shift map $\sigma : \Sigma \longrightarrow \Sigma$,
 to define Lyapunov exponents. It is well known that there are ergodic $\sigma$-invariant measures $\mu$ on $\Sigma$. Let $\mu$ be an {ergodic $\sigma$-invariant probability measure} on $\Sigma$.
The following is a consequence of Oseledets Multiplicative Ergodic Theorem:



\begin{thm} [A Consequence of Oseledets Multiplicative Ergodic Theorem]

There exist real numbers $\lambda_1 >  0 >  -\lambda_1$ 
and one-dimensional vector subspaces $E^u(x)$ and  $E^s(x)$ of $T_x(\partial K)$, $x \in M_0$,
depending measurably on $R(x) \in \Sigma$ such that:

\begin{enumerate} [(i)]

\item $E^u(x) $ and $E^s(x) $
for almost all $x \in M_0$;

\item $D_x B ( E^u(x)) = E^u(B(x))$ and $D_x B (E^s(x)) = E^s(B(x))$
for almost all $x\in M_0$, and

\item For almost all $x \in M_0$ there exists
$$\lim_{n\to\infty} \frac{1}{n} \log \|D_x B^n (w)\| = \lambda_1$$

\noindent whenever $0 \neq w \in E^u(x) $.

\end{enumerate} 
\label{OSE}
\end{thm}

\noindent Here "for almost all $x$" means "for almost all $R(x)$" with respect to $\mu$. The numbers $\lambda_1 >  0 >  -\lambda_1$ are called {\it Lyapunov exponents}, while the invariant 
subspaces $E^u(x)$ and $E^s(x)$ are called {\it Oseledets subspaces}.
\bigskip


\subsection{Propagation of unstable manifolds for open billiards}

We describe a formula which is useful in getting estimates for $$\lim_{m\to\infty} \frac{1}{m}\log \|D_xB^m w\|,  \, \, (0 \neq w \in E^u(x), x \in M_0).$$

\bigskip

\noindent Let $M_0$ be the non-wandering set of the billiard ball map $B$ of an open billiard. Then
$$ {\Lambda} = \{ \Phi_t(x) : x\in M_0\:, \: t \in \mathbb{R}\} , $$  is the non-wandering set for the billiard flow $\Phi_t$. For $x\in {\Lambda}$ and a sufficiently small $\epsilon > 0$ let 
\[\widetilde{W}_\epsilon^s(x) = \{ y\in {\Lambda} : d (\Phi_t(x),\Phi_t(y)) \leq \epsilon \: \, \mbox{\rm for all }
\: t \geq 0 \; , \: d (\Phi_t(x),\Phi_t(y)) \to_{t\to \infty} 0\: \},\]
\[\widetilde{W}_\epsilon^u(x) = \{ y\in {\Lambda} : d (\Phi_t(x),\Phi_t(y)) \leq \epsilon \: \, \mbox{\rm for all }
\: t \leq 0 \; , \: d (\Phi_t(x),\Phi_t(y)) \to_{t\to -\infty} 0\: \} \]
be the (strong) {\it stable} and {\it unstable manifolds} of size $\epsilon$ {for the billiard flow}. Then $\widetilde{E}^u(x) = T_x \widetilde{W}_\epsilon^u(x)$ and $\widetilde{E}^s(x) = T_x \widetilde{W}_\epsilon^s(x)$. In a similar way one defines stable/unstable manifolds for the billiard ball map $B$.
For any $x = (q,v) \in M_0$ define
$$W_\epsilon^s(x) = \{ y\in M_0 : d (B^n(x),B^n(y)) \leq \epsilon \: \, \mbox{\rm for all }
\: n \in \mathbb{N} \; , \: d (B^n(x),B^n(y)) \to_{n\to \infty} 0\: \},$$
$$W_\epsilon^u(x) = \{ y\in M_0 : d (B^{-n}(x), B^{-n}(y)) \leq \epsilon \: \, \mbox{\rm for all }
\: n \in \mathbb{N} \; , \: d (B^{-n}(x),B^{-n}(y)) \to_{n\to \infty} 0\: \}.$$
In what follows we will just write $W^u(x)$ and $W^s(x)$ for $W^u_{\epsilon}(x)$ and
$W^s_{\epsilon}(x)$, assuming some appropriately chosen sufficiently small $\epsilon > 0$ is involved. Similarly for $\widetilde{W}^u$ and $\widetilde{W}^s$.

 \bigskip 
 
It is well-known that there is an one-to-one correspondence between the stable/unstable manifolds for the billiard ball map and these for the flow. Geometrically the easiest (and most convenient way) to describe this is as follows. 

\bigskip

\noindent Given $x = (q,v) \in M_0$ (so $q\in \partial{K}$ and $v \in \mathbb{S}^1$), and a small $0< r < t_1(x)$, set $y = (q+r v, v)$.
Then there is a 1-1 correspondence
$$\varphi :W^u (x) \longrightarrow \widetilde{W}^u (y)$$
such that $\varphi(z,w) = (z + t\, w, w)$ for all $(z,w) \in W^u (x)$, where $t = t(z,w) > 0$.
Similarly, there is a correspondence between $W^s (x)$ and $\widetilde{W}^s (y)$. Moreover
$$D\varphi (x) : T_xM_0 \longrightarrow T_y \Lambda$$
is so that $D\varphi(x) (E^u(x)) =\widetilde{E}^u(y)$ and $D\varphi(x) (E^s(x)) = \widetilde{E}^s(y)$.

\bigskip
It is known that $\widetilde{W}^u(y)$ has the form $\widetilde{W}^u(y) = \widetilde{Y}$, where 
$$\widetilde{Y} = \{ (p,\nu_Y(p)) : p\in Y\}$$
for some smooth curve $Y$ in $\mathbb{R}^2$ containing the point $y$ such that $Y$ is strictly convex with respect to the unit normal field $\nu_Y$, i.e. the curvature of $Y$ is strictly positive.

\bigskip

Next, let $x$ and $y$ be as above and let $x_1 = (q_1,v_1) = B(x)$. Then $q_1 = q + t_1 \, v.$ Define $y_1 = (q_1+ r' v_1, v_1)$ for some small $0< r'< t_2(x)-t_1(x)$, where $0=t_0(x)<t_1(x)<t_2(x)$. Then there is a
1-1 correspondence
$$\varphi_1 :W^u (x_1) \longrightarrow \tilde{W}^u(y_1)$$
defined as above. Again, we can write $\widetilde{W}^u(y_1) = \widetilde{Y}_1$, where 
$$\widetilde{Y}_1 = \{ (p_1,\nu_Y(p_1)) : p_1\in Y_1\}$$
for some smooth curve $Y_1$ in $\mathbb{R}^2$ containing the point $y_1$ such that $Y_1$ is strictly convex with respect to the unit normal field $\nu_{Y_1}$. Moreover the following diagram is commutative, where $t = t_1+r'$:

$$\def\normalbaselines{\baselineskip20pt\lineskip3pt \lineskiplimit3pt}
\def\mapright#1{\smash{\mathop{\longrightarrow}\limits^{#1}}}
\def\mapdown#1{\Big\downarrow\rlap{$\vcenter{\hbox{$\scriptstyle#1$}}$}}
\begin{matrix}
W^u(x) &\mapright{B}& W^u(x_1)\cr 
\mapdown{\varphi}& & \mapdown{\varphi_1}\cr \widetilde{W}^u(y) = \widetilde{Y}&\mapright{\Phi_t}& \widetilde{W}^u(y_1) = \widetilde{Y}_1
\end{matrix}
$$ 
Similarly, the following diagram  is commutative:  
$$\def\normalbaselines{\baselineskip20pt\lineskip3pt \lineskiplimit3pt}
\def\mapright#1{\smash{\mathop{\longrightarrow}\limits^{#1}}}
\def\mapdown#1{\Big\downarrow\rlap{$\vcenter{\hbox{$\scriptstyle#1$}}$}}
\begin{matrix}
E^u(x) &\mapright{DB(x)}& E^u(x_1)\cr 
\mapdown{D\varphi}& & \mapdown{D\varphi_1}\cr \widetilde{E}^u(y) &\mapright{D\Phi_t(y)}& \widetilde{E}^u(y_1)
\end{matrix}
$$ 

\noindent Since the derivatives $D\varphi$ and $D\varphi_1$ are uniformly bounded, the above conjugacy can be used later
to calculate the Lyapunov exponents of the billiard ball map using propagation of appropriate convex curves $Y$ which we describe as follows. 

\bigskip 
 
Let $x_0 = (q_0,v_0) \in M_{0}$ and let $W^u_\epsilon (x_0)$ be the local unstable manifold for $x_0$ for sufficiently small $\epsilon > 0$. 
Let $t_1(x_0)$ be the time of the first reflection of $x_0$. Then $\widetilde{X}= {W}^u_\epsilon(x_0)  = \{ (q, n_X(q)) : q\in X\}$  for some $\mathcal{C}^3$ curve $X$ in $\Omega$ such that $q_0 \in X$ and $X$ is strictly convex curve with respect to the outer unit normal field $n_X (q)$. Let $X$ be parametrized by $q(s), s \in [0, a]$, such that $q(0) = q_0 $, and has unit normal field $n_X(q(s))$.  Set $q_0(s)= q(s)$. Let $q_j(s), j \geq 1$ be the $j$th-reflection points of the forward billiard trajectory $\gamma(s)$ generated by $x(s) = (q(s), n_X(q(s))$. We assume that $a > 0$ is sufficiently small so that the $j$th-reflection points $q_j(s)$ belong to the same boundary component $\partial{K}_{\xi_j}$ for every $s\in [0,a]$. Let $0 = t_0(x(s)) < t_1(x(s)) < ... < t_{m+1}(x(s))$ be the times of the reflections of the ray $\gamma(s)$ at $\partial K$.
 Let $\kappa_j(s)$ be the {\it curvature} of $\partial K_{\xi_j}$ at $q_j(s)$ and $\phi_j(s)$ be the {\it collision angle} between the outward unit  normal to $\partial K$ and the reflection ray of $\gamma(s)$ at $q_j(s)$. Also, let $d_j(s)$ be the {\it distance} between two reflection points i.e. $d_j(s) =  \| q_{j+1}(s) - q_j(s)\|$, $j=0,1,\dots,m$. 
 
\bigskip
 
\noindent Given a large $m\geq 1$, let $t_{m}(x(s)) < t < t_{m+1}(x(s))$. 
Set $\Phi_{t} (\widetilde{X}) = \widetilde{X}_{t}$. Let $\pi (\Phi_{t} (x(s))) = p(s)$. Then $p(s), s\in [0, a]$, is a parametrization of the $\mathcal{C}^3 $ curve $X_{t} = \pi (\Phi_{t} (\widetilde{X})$.


\bigskip 



Next, let $k_0(s) > 0$ be the {\it curvature} of $X$ at $q(s)$. Let $t_{j}(x(s)) < \tau < t_{j+1}(x(s))$, $j=1,2,\dots,m$. Denote by $u_\tau(s)$ be the shift of $(q(s), n(q(s)))$ along the forward billiard trajectory $\gamma(s)$ after time $\tau > 0$. Then $X_\tau=\{u_\tau(s) : s\in [0,a]\}$ is a $\mathcal{C}^3$ convex curve with respect to the outward normal field $n(u_\tau(s))$. Let $k_j(s) > 0$ be the curvature of $X_{t_j} = \lim_{\tau \searrow t_{j}(s)} X_\tau$  at $q_j(s)$. 
 It follows from Sinai \cite{Si1} that

\bigskip

\begin{equation}
k_{j+1}(s) = \frac{k_j(s)}{ 1 + d_j(s) k_j(s) } + 2 \frac{\kappa_{j+1}(s)}{\cos\phi_{j+1}(s)}\:\:\:\: , \:\:\: 0\leq j \leq m-1\;.
\label{k}
\end{equation}

\noindent Moreover, the curvature of $X_\tau$ at $u_\tau(s)$ is  

\begin{equation}
k_{\tau}(s) = \frac{k_j(s)}{ 1 + (\tau-t_j(s)) k_j(s) }.
\label{kt}
\end{equation}

\bigskip

Set
\begin{equation}
\delta_j(s) = \frac{1}{ 1 + d_j(s) k_j(s)} \:\:\:\: , \:\:\: 1\leq j \leq m\;.
\end{equation}

\bigskip

\noindent
\begin{thm} {\em{\cite{St2}}}
For all $s \in [0,a]$  we have

\begin{equation}
\| \dot{q}(s)\| = \| \dot{p}(s)\| \delta_{1}(s)\delta_{2}(s) \ldots \delta_m(s)\;.
\label{for.} 
\end{equation}

\label{ST theorem}

\end{thm}

\label{STd}

\noindent This was proved in \cite{St2} in the $2D$ case and in \cite{St3} in the general case. 



\bigskip


\noindent 

\noindent Finally, we want to introduce some notation related to the maximum and minimum of previous billiard characteristies $d_j(s)$,$\kappa_j(s)$, $\phi_j(s)$ and $k_j(s)$. For all $j$, we have $d_{\min} \leq d_j(s) \leq  d_{\max}$, where $d_{\max}$ and $d_{\min}$ are constants independent of $j$ such that $d_{\max} = \max \{d(K_i, K_k)\}$ and $d_{\min} = \min \{d(K_i, K_k)\}$ for $i\neq k$. Also, since $\partial K$ is strictly convex, we have constants $\kappa_{\min} >0$ and $\kappa_{\max} >0$ independent of $j$ such that $\kappa_{\min} \leq \kappa_j(s) \leq  \kappa_{\max}$. And it follows from the condition ({\bf{H}}) that there exists a constant $\phi_{\max} \in (0, \frac{\pi}{2})$ such that 
 $0 \leq \phi_j(s) \leq \phi_{\max} < \frac{\pi}{2}$, (see e.g. \cite{St1}). Let $k_j(s)$ be as in equation (\ref{k}). It follows easily that $k_{\min} \leq k_j(s) \leq  k_{\max}$, where $k_{\min} = 2\kappa_{\min}$ and $k_{\max} = \frac{1}{d_{\min}}  + \frac{2\kappa_{\max}}{\cos\phi_{\max}}$.

\section {Estimation of the largest Lyapunov exponent for open billiards}

A formula for the largest Lyapunov exponents for a rather general class of billiards can be found in \cite {ChM}, see Theorem 3.41 there. In our case we derive this formula again (see (\ref{g}) below) and then we use Theorem \ref{ST theorem} to derive important regularity properties of the largest Lyapunov exponent.

Assume that $\mu$ is an ergodic $\sigma$-invariant measure on $\Sigma$,
and let $x_0 = (q_0,v_0) \in M_0$ correspond to a typical point in $\Sigma$  with respect to $\mu$ via the
representation map $R$. That is as in Theorem \ref{OSE}, we have
$$\lambda_1 = \lim_{m\to\infty} \frac{1}{m} \log \|D_{x_0}B^m(w)\| ,$$ with $0 \neq w \in E^u(x_0)$.
As in Sect. \ref{STd}, let $X$ be a (small) $C^3$ strictly convex curve containing
$q_0 $ and having a unit normal field $n_X$ so that $n_X(q_0) = v_0$.
As in Sect. \ref{STd} again, let $X$ be parametrised by arc length via $q(s)$, $s \in [0,a]$, such that
$q(0) = q_0$. Let again $q_j(s)$, $j = 1,2, \ldots, m+1$, be the consecutive reflection points of the billiard 
trajectory $\gamma(s)$ determined by $x(s) = (q(s), n_X(q(s))$. Given an integer $m > 0$ and assuming
the interval $[0,a]$ is sufficiently small, the $j$th reflection points $q_j(s)$ belong to the same boundary
component $\partial K_{\xi_j}$ for all $s\in [0,a]$. Next, define $d_j(s)$, $t_j(x(s))$, etc. as in Sect. \ref{STd},
let $t_m(x(0)) < t < t_{m+1}(x(0))$, and let $p(s)$ be the parametrisation of $\widetilde{X}_t$ corresponding to $q(s)$.
Then the formula (\ref{for.}) in Theorem \ref {ST theorem} (holds with ${\|\dot{q}(s)\| =1}$ 
from our assumptions). Now the discussion 
in Sect. \ref{STd} implies that there exist some global constants $c_1 > c_2 > 0$, independent of $x_0$, $X$, $m$, etc. such that
$$c_2 \|\dot{p}(s)\| \leq \|D_{x_0}B^m(w)\| \leq c_1 \|\dot{p}(s)\|$$
for all $s\in [0,a]$. So, by (\ref{for.}),
$$\frac{c_2 }{\delta_1(0) \delta_2(0) \ldots  \delta_{m}(0)} \leq \|D_{x_0}B^m(w)\| 
\leq \frac{c_1}{\delta_1(0) \delta_2(0) \ldots \delta_{m}(0)}$$
for all $s \in [0,a]$. Using this for $s = 0$, taking logarithms and limits as $m \to \infty$, we obtain\\

\begin{equation*}
\begin{aligned}
- \lim_{m\to\infty} \frac{1}{m} \log \left( \delta_1(0) \delta_2(0) \ldots  \delta_{m}(0)\right) 
&\leq \lim_{m\to\infty}\frac{1}{m} \log \|D_{x_0}B^m(w)\| \\
&\leq - \lim_{m\to\infty} \frac{1}{m} \log \left( \delta_1(0) \delta_2(0) \ldots  \delta_{m}(0)\right).
\end{aligned}
\end{equation*}

\noindent Hence,

\begin{equation*}
\begin{aligned}
\lambda_1 &=\lim_{m\to \infty} - \frac{1}{m} \sum_{i=1}^{m} \log \delta_{i}(0). \\
%
%
\end{aligned}
\end{equation*}

\noindent This implies that the largest Lyapunov exponent at the initial point $x_0$, so at almost every point wilt respect to the given measure $\mu$, is given by
 
\begin{equation}
\begin{aligned}
\lambda_1=\lim_{m\to \infty}  \frac{1}{m} \sum_{i=1}^{m} \log {\Big(1+d_i (0) k_i (0)\Big)}. 
\label{g}
\end{aligned}
\end{equation}

\noindent From equation ({{\ref{g}}}), we can estimate the largest Lyapunov exponent from below and above as 

\begin{equation*}
\begin{aligned}
 \log {(1+d_{\min} k_{\min})} \leq  \lambda_1 \leq \log {(1+d_{\max} k_{\max})}.
\end{aligned}
\end{equation*}

\label{LE}





\section{Billiard deformations}
\setcounter{equation}{0}

\label{defo}

In this section, we consider some changes to the billiards in the plane, such as moving, rotating, and changing the shape of one or multiple obstacles. This kind of billiard transformation is called a billiard deformation as defined in \cite{W2}. We describe this deformation by adding an extra parameter $\alpha \in [0, b]$ for some $b \in \mathbb{R}^+$, which is called the deformation parameter, to the parametrization of the boundary of obstacles i.e., 
if the boundary of an obstacle parametrized by $\varphi(u)$, it will become $\varphi(u, \alpha)$. In this section, we provide the definition a billiard deformation as defined in \cite{W2}. In addition, we describe the propagation of unstable manifolds for billiard deformations. We also estimate the higher derivatives of some of the billiard characteristics such as distance, collision angle and curvature, with respect to deformation parameter $\alpha$. 



\bigskip

Let $\alpha \in I = [0, b]$, for some $b \in \mathbb{R}^+$, be a {\it deformation parameter} and let $\partial K_i(\alpha)$ be parametrized counterclockwise by $\varphi_i(u_i, \alpha)$ and parametrized by arc-length $u_i$. Let $q_i = \varphi_i(u_i, \alpha) $ be a point that lies on $\partial K_i(\alpha)$. Denote the perimeter of $\partial K_i(\alpha)$ by $L_i(\alpha)$, and let $P_i=\{ (u_i, \alpha) : \alpha \in I, u_i \in [0,L_i(\alpha)]\}$.

\bigskip

\begin{defn} \cite{W2} For any $\alpha \in I = [0, b]$, let $K(\alpha)$ be a subset of $\mathbb{R}^2$. For integers $r \geq 4, r' \geq 2$, we call $K(\alpha) $ a $\mathcal {C}^{r,r'}$-{{\it billiard deformation}} (i.e. $\mathcal{C}^r$ with respect to $u$ and $\mathcal{C}^{r'}$ with respect to $\alpha$) if the following conditions hold for all $\alpha \in I$:

\begin{enumerate}

\item $K(\alpha) = \bigcup_{i=1}^{z_0} K_i(\alpha)$ satisfies the no-eclipse condition $({\bf{H}})$.  

\item Each $K_i(\alpha)$ is a compact, strictly convex set with $\mathcal{C}^r$ boundary and total arc length $L_i(\alpha)$. 

\item Each $K_i$ is parametrized counterclockwise by arc-length with $\mathcal{C}^{r,r'}$ functions ${\varphi}_i : P_i \to \mathbb{R}^2 .$

\item For all integers $0 \leq l \leq r, 0\leq l' \leq r'$ (apart from $l=l'=0$), there exist constants $C_\varphi ^{(l,l')} $ depending only on the choice of the billiard deformation and the parametrizations $\varphi_i$, such that for all integers $i=1,2,3,...,z_0$, 
$$  \Big | \frac { \partial ^{l+l'} \varphi_i}{\partial u^l_i \partial \alpha^{l' }} \Big | \leq {C}^{(l,l')}_ \varphi. $$

\end{enumerate}

\label{43}
\end{defn}


\noindent Let $B_\alpha$ be the open billiard map on a non-wandering set $M_\alpha$ for $K(\alpha)$. Let $\Sigma$ defined in Sec. \ref{coding}, we defined $R_\alpha : M_\alpha \to \Sigma$ by $R_\alpha (x(\alpha)) = \xi(x(\alpha))$. We can write the points that correspond to the billiard trajectories according to the parameterization in previous definition as follows, $\pi(B^j(x(\alpha)) ) = q_{\xi_j} (\alpha) = \varphi_{\xi_j} (u_{\xi_j} (\alpha), \alpha) \in \partial K_{\xi_j}(\alpha)$, where $u_{\xi_j}(\alpha) \in[0, L_{\xi_j}(\alpha)]$. For brevity, we will write $q_j (\alpha) = \varphi_j (u_j (\alpha), \alpha)$.


\bigskip 

The next corollary shows that $u_j(\alpha) = u_{\xi_j}(\alpha)$ for a fixed $\xi \in \Sigma$, is differentiable with respect to $\alpha$. This corollary is proved in \cite{W2}. 

\begin{thm} {\em{\cite{W2}}}
Let $K(\alpha)$ be a $\mathcal{C}^{r,r'}$ billiard deformation with $r,r' \geq 2$. Then $u_j(\alpha)$ is $\mathcal{C}^{min\{r-1,r'-1\}}$ with respect to $\alpha$, and there exist constants $C_u^{(n)} >0$ such that $$\Big | \diff[n]{u_j(\alpha)}{\alpha} \Big |  \leq C^{(n)}_u.$$

\label{3}

\end{thm}

The next corollary follows from  {{Definition \ref{43}}} and Theorem {\ref {3}}. 
\begin{cor}
Let $K(\alpha)$ be a $\mathcal{C}^{r,r'}$ billiard deformation with $r,r' \geq 2$. Let $q_j(\alpha)$ belongs to $\partial K_{\xi_j}$. Then $q_j(\alpha)$ is $\mathcal{C}^{n}$, where $n=\min\{r-1,r'-1\}$, with respect to $\alpha$, and there exist constants $C_q^{(n)} >0$ such that $$\Big | \diff[n]{q_j(\alpha)}{\alpha} \Big |  \leq C^{(n)}_q.$$

\label{Dq}
\end{cor}


\subsection{Propagation of unstable manifolds for billiard deformations}

We described the unstable manifolds propagation in Section \ref {STd} for open billiards. Here in this section, we describe it for billiard deformations. 

Let $K(\alpha), \alpha \in [0,b]$ be a $\mathcal{C}^{r,r'}$ billiard deformation as in {{Definition \ref{43}}} with $r \geq 3 ,r' \geq 1$. $x_0(\alpha) = (q_0(\alpha),v_0(\alpha)) \in M_{\alpha}$ and let $W^u_\epsilon (x_0 (\alpha))$ be the local unstable manifold for $x_0(\alpha)$ for sufficiently small $\epsilon > 0$. Take a curve $X_\alpha$ containing $q_0(\alpha)$ such that $X_\alpha = \{q_0(s, \alpha) : s \in [0,a]\}$ is a convex curve with outer unit normal field $n_X(q_0(s, \alpha)) = v_0(\alpha)$ and $\mathcal{C}^3$ with respect to $s$. It follows from Sinai \cite{Si1}, \cite{Si2} that $W^u_\epsilon (x_0 (\alpha)) =  \{ (q_0(s, \alpha), n_X(q_0)): s \in [0,a]\}$. Set $\widetilde{X}_{\alpha} =W^u_\epsilon (x_0 (\alpha))$. Let $a\in \mathbb{R}^+$ be small enough such that all reflection points $q_j(s, \alpha), j= 1,2,..., m$, that are generated by $x_0(s,\alpha)=(q_0(s,\alpha), n_X(q_0(s,\alpha)))$ belong to the same boundary $\partial K _{\xi_j}(\alpha)$. Let $d_j(s,\alpha) = \| q_{j+1}(s,\alpha) - q_{j}(s,\alpha) \|$ be the distance between two reflection points $q_{j+1}(s,\alpha)$ and $q_{j}(s,\alpha)$. Denote the curvature of $\partial K(\alpha)$ at $ q_{j}(s,\alpha)$ by $\kappa_j(s, \alpha)$, the collision angle between the unit normal to $\partial K(\alpha)$ and the reflection vector at $q_j(s,\alpha)$ by $\phi_j(s,\alpha)$, and the curvature of $X$ at $q_0(s,\alpha)$ by $k_0(s,\alpha)$ .

Let $t_j(x(s,\alpha))= t_j(s, \alpha)$ be the time of the j-th reflection. Given $t$ with $t_j < t < t_{j+1}$ for some $j = 1, 2,..., m$,  set $\pi(\Phi_t (\widetilde{X}_{\alpha})) = X_{\alpha_t}$. Then $X_{\alpha_t} =\{u_{\alpha_t}(s,\alpha) : s\in[0,a]\}$ is $\mathcal{C}^3$ with respect to $s$ and a convex curve with outer unit normal field $n_{X_{\alpha_t}}(u_{\alpha_t}(s, \alpha))$. Denote the curvature of $X_{{\alpha_t}_j(s,\alpha)}$ at $q_j(s,\alpha)$ by $k_{j}(s,\alpha)$, where $X_{{\alpha_t}_j(s,\alpha)} = \lim_{t \searrow t_{j}(s,\alpha)} X_{\alpha_t}$. As in equation (\ref{k}), we can define $k_{j}(s,\alpha)$  as follows:

\bigskip

\begin{equation}
k_{j+1}(s,\alpha) = \frac{k_j(s,\alpha)}{ 1 + d_j(s,\alpha) k_j(s,\alpha) } + 2 \frac{\kappa_{j+1}(s,\alpha)}{\cos\phi_{j+1}(s,\alpha)}\:\:\:\: , \:\:\: 0 \leq j \leq m-1\;.
\label{k(alpha)}
\end{equation}





\bigskip

From now on, we will need to use previous characteristics in the case $s=0$, so for brevity, we will write $d_j(\alpha) = d_j(0, \alpha), etc$. Also, we denote the billiard deformation by $K(\alpha)$, so all of its characteristics will be denoted $d_j(\alpha), k_j(\alpha)$, etc. The initial open billiard is $K(0)$ so all of its characteristics will be denoted $d_j(0)$, etc.







\label{alpha}

\subsection{The higher derivatives of billiard characteristics}

\label{higher}

Let $K(\alpha)$ be a $\mathcal{C}^{r,r'}$ billiard deformation as in {{ Definition \ref{43}}} with $r \geq 4 ,r' \geq 2$. Recall that $\partial K_{\xi_j}(\alpha)$ is parametrized by arc-length $u_j$and $q_j(\alpha)= \varphi_j(u_j(\alpha),\alpha) \in \partial K_{\xi_j}(\alpha)$. Here, we state some corollaries related to bounds of the higher derivatives of curvature, distance and collision angle of a billiard deformation. These corollaries are forthright consequences of condition 4 in {{Definition \ref{43}}}.


\begin{cor}
Let $K(\alpha)$ be a $\mathcal{C}^{r,r'}$ billiard deformation with $r \geq 4 ,r' \geq 2$. Then the curvature $\kappa_j(\alpha)$ at $q_j(\alpha)$ is $\mathcal{C}^{n}$, where $n = \min\{r-3,r'-1\}$ with respect to $\alpha$ and there exist constants 
$C _\kappa ^{(n)} > 0 $ depending only on $n$ such that 
$$ \Big | \diff[n]{\kappa}{\alpha}\Big | \leq C_\kappa ^ {(n)}.$$

\label{kappa}

\end{cor}


\begin{proof}[Proof.]
Suppose $K(\alpha)$ is a a $\mathcal{C}^{r,r'}$ billiard deformation with $r \geq 3 ,r' \geq 1$. Since $\partial K_j(\alpha)$ is paramitrized by arc-length $u_j$, then the curvature of $\partial K_j(\alpha)$ at $q_j(\alpha) = \varphi_j(u_j(\alpha), \alpha)$ is $\kappa_j=\frac{\partial ^2 \varphi_j}{\partial u_j^2}$, for $j=0,1,...,m$. Then $\kappa_j(\alpha)$ is $\mathcal{C}^{\min\{r-3,r'-1\}}$ with respect to $\alpha$.

\bigskip

For the first derivative, we have 

\begin{equation*}
\begin{aligned}
\Big |\frac{d\kappa_j}{d\alpha} \Big|&= \Big|\frac{\partial^3 \varphi_j}{\partial u^3_j} \frac{\partial u_j}{\partial \alpha} + \frac{\partial^3 \varphi_j }{\partial u^2_j \partial \alpha}\Big| \leq C_{\kappa}^{(1)},\\
\end{aligned}
\end{equation*}

\noindent this estimate was obtained in \cite{W2}. Next, we continue to estimate the second derivative, so we have  

\begin{equation*}
\begin{aligned}
\Big |\frac{d^2\kappa_j}{d\alpha^2} \Big|&= \Big |\frac{\partial^4 \varphi_j}{\partial u^4_j} \big (\frac{ \partial u_j}{\partial \alpha} \big)^2 + \frac{\partial^3 \varphi_j}{\partial u^3_j} \frac{\partial u^2_j}{\partial \alpha^2}    + 2\frac{\partial^4 \varphi_j }{\partial u^3_j \partial \alpha }  \frac{ \partial u_j}{\partial \alpha} +\frac{\partial^4 \varphi_j }{\partial u^2_j \partial \alpha^2 }  \Big|.\\
  \end{aligned}
\end{equation*} 

\bigskip

\noindent By using condition 4 in {{Definition \ref{43}}} and Theorem \ref{3}, there exists a constant $C_{\kappa}^{(2)}>0$ such that 

\bigskip

\begin{equation*}
\begin{aligned}
\Big |\frac{d^2\kappa_j}{d\alpha^2} \Big| &\leq  {C}_\varphi^{(4,0) } ({C}_u^{(1)})^2 + {C}_\varphi^{(3,0) }{C}_u^{(2)}  + 2 {C_\varphi}^{(3,1)} {C}_u^{(1)} +  {C}_\varphi^{(2,2)} = C_{\kappa}^{(2)}.
\end{aligned}
\end{equation*}

\noindent  Continuing by induction we see that the $n$-th derivative, where $n = \min\{r-3, r'-1\}$, is bounded by a constant $C_\kappa^{(n)}> 0$ which depends only on $n$ such that $ \Big | \diff[n]{\kappa}{\alpha}\Big | \leq C_\kappa ^ {(n)}.$
\end{proof}


\begin{cor}
Let $K(\alpha)$ be a $\mathcal{C}^{r,r'}$ billiard deformation with $r \geq 3 , r' \geq 1$. Then the distance $d_j (\alpha)$ between two points $q_{j+1}(\alpha)$ and $ q_{j}(\alpha) $ is $\mathcal{C}^{n}$, where $n = \min\{r-1,r'-1\}$ with respect to $\alpha$ and there exist constants 
$C _d ^{(n)} > 0$ depending only on $n$ such that 
$$ \Big | \diff[n]{d_j}{\alpha} \Big | \leq C_d ^ {(n)}.$$

\label{distance}

\end{cor}

\begin{proof}[Proof.]
Since $d_j= \|q_{j+1}(\alpha) - q_{j}(\alpha)\| = \| \varphi_{j+1}(u_{j+1}(\alpha), \alpha) - \varphi_{j}(u_j(\alpha), \alpha) \|$ for $j=0,1,...,m$, 
then $d_j$ is $\mathcal{C}^{\min\{r-1, r'-1\}}$. The first derivative is 

\begin{equation*}
\begin{aligned}
 \diff{d_j}{\alpha} &= \Big < \frac{   \varphi_{j+1}(u_{j+1}(\alpha), \alpha) - \varphi_{j}(u_j(\alpha), \alpha) } { \| \varphi_{j+1}(u_{j+1}(\alpha), \alpha) - \varphi_{j}(u_j(\alpha), \alpha) \| }                     , \frac{\partial \varphi_{j+1} }{\partial u_{j+1}} \frac{ \partial u_{j+1}}{\partial \alpha} +  \frac{\partial\varphi_{j+1}}{\partial \alpha}+ \frac{\partial \varphi_{j} }{\partial u_{j}} \frac{ \partial u_{j}}{\partial \alpha} + \frac{\partial\varphi_{j}}{\partial \alpha} \Big > .
\end{aligned}
\end{equation*}

\noindent And then 

\begin{equation*}
\begin{aligned}
\Big | \diff{d_j}{\alpha} \Big |&= \Big |\frac{\partial \varphi_{j+1} }{\partial u_{j+1}} \frac{ \partial u_{j+1}}{\partial \alpha} +  \frac{\partial\varphi_{j+1}}{\partial \alpha}+ \frac{\partial \varphi_{j} }{\partial u_{j}} \frac{ \partial u_{j}}{\partial \alpha} + \frac{\partial\varphi_{j}}{\partial \alpha} \Big |  \leq C_{d}^{(1)},
\end{aligned}
\end{equation*}

\noindent which was estimated in \cite{W2}. For the second derivative, using condition 4 in {{Definition \ref{43}}} and Theorem \ref{3} it follows that

\begin{equation*}
\begin{aligned}
\Big | \diff[2]{d_j}{\alpha} \Big |&= \Big |\frac{\partial^2 \varphi_{j+1} }{\partial u^2_{j+1}} \Big(\frac{ \partial u_{j+1}}{\partial \alpha} \Big )^2  + \frac{\partial \varphi_{j+1}}{\partial u_{j+1}} \frac{ \partial^2u_{j+1}}{\partial \alpha^2}  +2\frac{\partial^2 \varphi_{j+1} }{\partial u_{j+1} \partial \alpha }  \frac{ \partial u_{j+1}}{\partial \alpha} 
+ \frac{\partial^2 \varphi_{j+1}}{\partial \alpha^2} \\
&\, \, \, \, \, \, +\frac{\partial^2 \varphi_{j} }{\partial u^2_{j}} \Big(\frac{ \partial u_{j}}{\partial \alpha} \Big )^2  + \frac{\partial \varphi_{j}}{\partial u_{j}} \frac{ \partial^2u_{j}}{\partial \alpha^2}  +2\frac{\partial^2 \varphi_{j} }{\partial u_{j} \partial \alpha }  \frac{ \partial u_{j}}{\partial \alpha} 
+ \frac{\partial^2 \varphi_{j}}{\partial \alpha^2} \Big |.\\
\end{aligned}
\end{equation*}

\noindent By using condition 4 in {{Definition \ref{43}}} and Theorem \ref{3}, there exists a constant $C_{\kappa}^{(2)}>0$ such that

\begin{equation*}
\begin{aligned}
\Big | \diff[2]{d_j}{\alpha} \Big | &\leq 2 {C}_\varphi^{(2,0) } ({C}_u^{(1)})^2 + 2{C}_u^{(2)} + 4 {C}_\varphi^{(1,1) } ({C}_u^{(1)})^2 +2 {C}_\varphi^{(0,2) }\\
%
%
 &= C_{d}^{(2)}.
\end{aligned}
\end{equation*}

\noindent Continuing by induction, we can see that there exists a constant $C_d^{(n)} >0 $ depends only on $n$ such that $ \Big | \diff[n]{d_j}{\alpha} \Big | \leq C_d ^ {(n)}.$
\end{proof}


\begin{cor}
Let $K(\alpha)$ be a $\mathcal{C}^{r,r'}$ billiard deformation with $r \geq 4 , 
r' \geq 2$. Then $\cos \phi_j (\alpha)$ is $\mathcal{C}^{\min\{r-1, r'-1\}}$ and there exists a constant $C _\phi ^{(n)} > 0$ depending only on $n$ such that 
$$ \Big | \diff[n]{\cos\phi_j}{\alpha} \Big | \leq C_{\phi} ^ {(n)}.$$

\label{angle}

\end{cor}

\begin{proof}[Proof.]

We can write 

\begin{equation*}
\begin{aligned}
\cos {2\phi_j} &= \frac{\big (q_{j+1} (\alpha)- q_j(\alpha) \big ) \cdot \big (q_{j}(\alpha) - q_{j-1} (\alpha)\big)} {|q_{j+1}(\alpha) - q_j (\alpha)| |q_{j}(\alpha) - q_{j-1} (\alpha)|}\\
&= \frac{\big (\varphi_{j+1} (,u_{j+1}, \alpha)- \varphi_j (u_j, \alpha) \big ) \cdot \big (\varphi_{j}(u_j, \alpha) - \varphi_{j-1} (u_{j-1}\alpha)\big)} {|\varphi_{j+1}(u_{j+1}, \alpha) - \varphi_j (u_j, \alpha)| |\varphi_{j}(u_j, \alpha) - \varphi_{j-1} (u_{j-1}, \alpha)|}.
\end{aligned}
\end{equation*}

\noindent And then, $\cos\phi_j (\alpha)= \sqrt {\frac{\cos {2\phi_j(\alpha)}  +1}{2}}$. Therefore, the statement follows from condition 4 in {{Definition \ref{43}}} and Corollary \ref{3}.
\end{proof}


\noindent The next corollary follows from Corollaries \ref{kappa}, \ref{angle}.

\begin{cor}
Let $K(\alpha)$ be a $\mathcal{C}^{r,r'}$ billiard deformation with $r \geq 4 ,r' \geq 2$. Then the expression  $g_j (\alpha)= \frac{2\kappa_j}{\cos\phi_j}$ is $\mathcal{C}^{\min\{r-3,r'-1\}}$ and there exist constants $C _g ^{(n)} > 0$ depending only on $n$ such that $$ \Big | \diff[n]{g_j}{\alpha}\Big | \leq C_g ^ {(n)}.$$

\label{gamma}

\end{cor}



\noindent The next corollary concerning the curvature $k_j$, defined in (\ref{k(alpha)}), follows from Corollaries \ref{distance} and \ref{gamma}.

\begin{cor}
Let $K(\alpha)$ be a $\mathcal{C}^{r,r'}$ billiard deformation with $r \geq 4 ,r' \geq 2$. Then the curvature  $k_j(\alpha)$ is $\mathcal{C}^{n}$, where $n = \min\{r-3,r'-1\}$ and here exist constants $C_k^{(n)}$ depending only on $n$ such that
$$ \Big | \diff[n]{k_j}{\alpha}\Big | \leq C_k ^ {(n)}. $$

\label{k_j}

\end{cor}


\begin{proof}[Proof.]

First, we recall 

\begin{equation*}
k_{j+1}(\alpha) = \frac{k_j(\alpha)}{ 1 + d_j(\alpha) k_j(\alpha)} + 2 \frac{\kappa_{j+1}(\alpha)}{cos\phi_{j+1}(\alpha)} \:\:\:\: , \:\:\: 0\leq j \leq m-1\;.
\end{equation*}

\noindent We will write $k_{j+1}(\alpha)$ simply as follows 

\begin{equation*}
k_{j+1} (\alpha)= \frac{k_j(\alpha)}{ 1 + d_j(\alpha) k_j(\alpha)} + g_{j+1}(\alpha),
\end{equation*}

\noindent where $g_{j+1}(\alpha) = \frac{2\kappa_{j+1}}{\cos\phi_{j+1}}$.  \cite{W2} contains an estimate that the first derivative of $k_j(\alpha)$ with respect to $\alpha$ is bounded by a constant $C_k^{(1)}$. Here, we use the same argument in \cite{W2} and show that the second derivative of $k_j(\alpha)$ with respect to $\alpha$ is also bounded. These estimates are useful and will be used later in Section \ref{Con.}.  

\bigskip

Next, we start with the first derivative of $k_{j+1}$ with respect to $\alpha$ and we will use the notation $\dot{k}$, $\ddot{k}$,...etc. to simplify equations. So, we have

\begin{equation*}
\begin{aligned}
\dot{k}_{j+1} &= \frac{\dot{k}_j } {(1+d_j k_j)^2 } - \frac{ \dot{d}_j k_j^2} { (1+d_j k_j)^2} + \dot{g}_{j+1} .\\
%
%
\end{aligned}
\end{equation*} 

\noindent And for the second derivative, we have
 
\begin{equation*}
\ddot{k}_{j+1} = \frac{\ddot{k}_j } {(1+d_j k_j)^2 } - \frac{k^2_j (\ddot{d}_j + \ddot{d}_j d_j k_j - 2 \dot{d}^2_j k_j) + 2\dot{k}_j(\dot{k}_j d_j + 2 \dot{d}_j k_j )} { (1+d_j k_j)^3}  + \ddot{g}_{j+1}.
\label{Dk}
\end{equation*}

%


\noindent Let 

\begin{equation*}
\begin{aligned} 
\beta_j &= \frac { 1} {(1+d_j k_j)^2 }, \\
\eta_j &=  - \frac{k^2_j (\ddot{d}_j + \ddot{d}_j d_j k_j - 2 \dot{d}^2_j k_j) + 2\dot{k}_j(\dot{k}_j d_j +2 \dot{d}_j k_j )} { (1+d_j k_j)^3}  + \ddot{g}_{j+1}  \:\:\:\: , \:\:\: 0\leq j \leq m-1\;.
\end{aligned}
\end{equation*}

\noindent From Corollaries \ref{distance} and \ref{gamma}, and the estimate of $\dot{k}_j$, we have

\begin{equation*}
\begin{aligned} 
|\beta_j |  \leq \beta_{\max} &= \frac { 1} {(1+d_{\min} k_{\min})^2 }, \\
%
 |\eta_j |   \leq \eta_{\max} & =\frac{k_{\max}^2(C_d^{(2)} + C_d^{(2)} d_{\max}k_{\max}  + 2 (C_d^{(1)})^2 k_{\max})} {(1+d_{\min} k_{\min} )^3 }\\
 &+ \frac{2C_k^{(1)}(C_k^{(1)} d_{\max} + 2C_d^{(1)}k_{\max})}{(1+d_{\min} k_{\min} )^3 }  + C_g ^{(2)}.
\end{aligned}
\end{equation*}

\noindent Then, we have

\begin{equation*}
\begin{aligned} 
\ddot{k}_m (\alpha)  &= \eta_{m-1} + \beta_{m-1} \ddot{k}_{m-1} (\alpha) \\
&= \eta_{m-1} + \beta_{m-1}\, \eta_{m-2} + .... + \beta_{m-1} .... \beta_{1}\,\eta_{0} + \beta_{m-1} .... \beta_{0} \, \ddot{k}_0(\alpha).   \\
\end{aligned} 
\end{equation*}


%



%

\noindent To solve this equation, we assume that $(q(\alpha), v(\alpha))$ is periodic such that $B_\alpha ^m(q(\alpha), v(\alpha))= (q(\alpha), v(\alpha))$. Then $k_m(\alpha) = k_0(\alpha)$. From this, we can solve the previous equation as follows

\begin{equation*}
\begin{aligned} 
\ddot{k}_m(\alpha) - \beta_{m-1} .... \beta_{0} \, \ddot{k} (\alpha) &= \eta_{m-1} + \beta_{m-1}\, \eta_{m-2} + .... + \beta_{m-1} .... \beta_{1}\,\eta_{0}  \\
\ddot{k}_m(\alpha) & = \frac{ 1} {1- \beta_{m-1} .... \beta_{0} } \Big ( \eta_{m-1} + \beta_{m-1}\, \eta_{m-2} + .... + \beta_{m-1} .... \beta_{1} \Big)\\
\end{aligned} 
\end{equation*}

\noindent By the maximum value of $\eta_j$ and $\beta_i$, we have 

\begin{equation*}
\begin{aligned} 
| \ddot{k}_m(\alpha) | & \leq \frac{ \eta_{\max}} {1- \beta^m_{\max} } \Big (  1+ \beta_{\max} + .... +\beta_{\max}^{m-1}   \Big)\\
&= \frac{ \eta_{\max}} {1- \beta^m_{\max} } \Big (\frac{1 - \beta^{m}_{\max}} {1- \beta_{\max} } \Big) \\
& = \frac{ \eta_{\max}} {1- \beta_{\max} }.
\end{aligned} 
\end{equation*}

\noindent This means there exists a constant $C_k^{(2)} > 0$ does not depend on $m$ or $\alpha$ such that $|\ddot{k}_j(\alpha) | \leq C_k^{(2)}$, for every $j =0,1, ..., m$. Continuing by induction we can see that the $n$-th derivative of $k_j(\alpha)$ with respect to $\alpha$ is bounded by constant $C_k^{(n)} >0$ that depending only on $n$. 
\end{proof}



%
%
%
%


%
%




\section{{\large{Continuity of the largest Lyapunov exponent}}}

\label{Con.}

In this section, we show that the largest Lyapunov exponent $\lambda_1$ depends continuously on a planar billiard deformation. Let $K(\alpha)$ be a billiard deformation as defined in {{Definition \ref{43}}} and let $K(0)$ be the initial open billiard. 
Let $k_j(\alpha), k_j(0)$ and $d_j(\alpha), d_i(0)$ be the curvatures and the distances that are described in section \ref{alpha}.  

\bigskip

For every $\alpha \in [0,b]$, let $M_\alpha$ be the non-wandering set for the billiard map
and let $R_\alpha : M_\alpha \longrightarrow \Sigma$ be the analogue of the conjugacy map $R: M_0\longrightarrow \Sigma$,
so that the following diagram  is commutative:
$$\def\normalbaselines{\baselineskip20pt\lineskip3pt \lineskiplimit3pt}
\def\mapright#1{\smash{\mathop{\longrightarrow}\limits^{#1}}}
\def\mapdown#1{\Big\downarrow\rlap{$\vcenter{\hbox{$\scriptstyle#1$}}$}}
\begin{matrix}
M_\alpha &\mapright{B_{\alpha}}& M_\alpha\cr 
\mapdown{R_\alpha}& & \mapdown{R_\alpha}\cr \Sigma &\mapright{\sigma}& \Sigma
\end{matrix}
$$ 
where $B_\alpha$ is the billiard ball map on $M_\alpha$. By Theorem \ref{OSE} there exists a subset $A_\alpha$ of $\Sigma$
with $\mu(A_\alpha) = 1$ so that 
\begin{equation}
\lambda_1(\alpha) = \lim_{m\to\infty} \frac{1}{m} \log \| D_{x_0}B^m_\alpha(w)\|
\label{B_alpha}
\end{equation}
for all $x \in M_\alpha$ with $R_\alpha(x) \in A_\alpha$. Similarly, let $A_0$ be the set with $\mu(A_0) = 1$
which we get from Theorem \ref{OSE} for $\alpha = 0$.

\bigskip

\begin{lem}
Given an arbitrary sequence 
\begin{equation*}
\alpha_1, \alpha_2, \ldots, \alpha_{p}, \ldots 
\end{equation*} 
of elements of $[0,b]$, for 
$\mu$-almost all $\xi\in \Sigma$ the formula {\rm{(\ref{B_alpha})}} is valid for $\alpha = \alpha_p$ and $x = R_{\alpha}^{-1}(\xi)$ for all $p = 1,2, \ldots$
and also for $\alpha = 0$ and $x = R^{-1}(\xi)$.

\label{alpha_tilde{k}}

\end{lem}

\begin{proof}[Proof.]

The set $A = A_0 \cap \cap_{p=1}^\infty A_{\alpha_{p}}$  has $\mu(A) = 1$ since
$$\Sigma\setminus A = (\Sigma \setminus A_0) \cup \cup_{{p}}^\infty (\Sigma\setminus A_{\alpha_{p}})$$
has measure zero as a countable union of sets of measure zero. If $\alpha = \alpha_p$ for some $p$
and $R_\alpha(x) \in A$, then $R_\alpha(x) \in A_{\alpha_{p}}$ so formula (\ref{B_alpha}) holds. Similarly (\ref{B_alpha}) holds for $\alpha = 0$ as well.
\end{proof}

Thus, using the notation $x(0,\alpha) \in M_\alpha$, we can choose $\xi \in \Sigma$ so that 
formula (\ref{B_alpha}) applies for $\alpha = \alpha_{p}$ and $x = x (0,\alpha_{p})$ for all ${p} = 1, 2, \ldots$, and also for $\alpha = 0$
and $x = (0,0)$.

\bigskip

From the formula for the largest Lyapunov exponent (\ref{g}), we can write the Lyapunov exponents for $K(\alpha)$ and $K(0)$ as follows: 

\begin{equation*}
\begin{aligned}
\lambda_1 (\alpha)&=\lim_{m\to \infty}  \frac{1}{m} \sum_{j=1}^{m} \log {\Big(1+d_j (\alpha) k_j (\alpha)\Big)}= \lim_{m\to \infty} \lambda^{(m)}_1 (\alpha), \\
 \lambda_1 (0)&=  \lim_{m\to \infty}  \frac{1}{m} \sum_{j=1}^{m} \log {\Big(1+d_j (0) k_j (0)\Big)}= \lim_{m\to \infty} \lambda^{(m)}_1 (0),
\end{aligned}
\end{equation*}

\noindent where

\begin{equation}
\begin{aligned}
\lambda^{(m)}_1 (\alpha)=  \frac{1}{m} \sum_{j=1}^{m} \log {\Big(1+d_j (\alpha) k_j (\alpha)\Big)}  \\
{\rm{and}} \, \, \,\, \,  
\lambda^{(m)}_1 (0)=   \frac{1}{m} \sum_{j=1}^{m} \log {\Big(1+d_j (0) k_j (0)\Big)}.\\
\label{lambda_k}
\end{aligned}
\end{equation}



\noindent Now, we prove {\bf {Theorem \ref{con.}}}

\begin{proof}[Proof of {\rm {Theorem \ref{con.}:}}]

Let $K(\alpha)$ be a $\mathcal{C}^{4, 2} $ billiard deformation in $\mathbb{R}^2$, and let \newline 
$\alpha \in [0, b]$. Assume that $\lambda_1(\alpha)$ is not continuous at $\alpha = 0$.  Then there exists $\varepsilon >0$ and a sequence $ \alpha_1 > \alpha_2 > ... > \alpha_{p} > ... \to 0$ in $[0, b]$ with $\alpha_p \to 0$ such that  $ | \lambda_1 ^ m (\alpha_k) - \lambda_1 ^ m (0)  | \geq \varepsilon $ for all $p \geq 1$.




\bigskip

\noindent  By using Lemma \ref{alpha_tilde{k}} and the previous expressions of $\lambda_1^{m}(\alpha)$ for $\alpha = \alpha_p$ and $\lambda_1^{m}(0)$ in (\ref{lambda_k}), we have


\begin{align*}
 \Bigg | \lambda_1 ^ m (\alpha_{p}) - \lambda_1 ^ m (0) \Bigg | &= \Bigg | \frac{1}{m} \sum_{j=1}^{m}(\log \delta_j (\alpha_{p})- \log\delta_j (0)) \Bigg | \\
&= \Bigg | \frac{-1}{m} \sum_{j=1}^{m}( \log (1+ d_j(\alpha_{p}) k_j(\alpha_{p})) - \log(1+ d_j(0) k_j(0))) \Bigg | \\
&\leq  \frac{1}{m} \sum_{j=1}^{m}  \Bigg |  \log (1+ d_j(\alpha_{p}) k_j(\alpha_{p})) - \log(1+ d_j(0) k_j(0))  \Bigg | \\ 
&\leq  \frac{1}{m} \sum_{j=1}^{m}  \Bigg |  \frac {1+ d_j(\alpha_{p}) k_j(\alpha_{p})- (1+ d_j(0)  k_j(0)) } {1+ \min\{d_j (\alpha_{p}) k_j (\alpha_{p}), d_j (0) k_j (0)\}} \Bigg | \\
&=  \frac{1}{m} \sum_{j=1}^{m}  \Bigg |  \frac { d_j(\alpha_{p}) k_j(\alpha_{p})-  d_j(0)  k_j (0)} {1+ d_{\min}  k_{\min} } \Bigg | \\
&=  \frac{1}{m} \, {C}_0 \sum_{j=1}^{m}  \Bigg | d_j(\alpha_{p}) k_j(\alpha_{p})- d_j (0) k_j (0)\Bigg | \\
&=  \frac{1}{m} \, {C}_0 \sum_{j=1}^{m}  \Bigg | (d_j(\alpha_{p}) - d_j (0)) k_j(\alpha_{p}) + d_j(0) (k_j (\alpha_{p} ) - k_j(0))\Bigg |, 
\end{align*}

\noindent where $C_0 = \frac{1}{1+d_{\min} k_{\min}} >0$ is a global constant independent of $\alpha_{p}$.

\bigskip 

\noindent Fix a small $\delta >0$; we will state later how small $\delta >0$ should be. Next consider $p$ sufficiently large so that $\alpha_p < \delta$. For all $ p$, we have $| k_j (\alpha_{p} ) - k_j(0)| = \alpha_{p}  |\dot{k}_j(s(\alpha_{p} ))|$ and $| d_j (\alpha_{p} ) - d_j(0)| = \alpha_{p}  |\dot{d}_j(r(\alpha_{p} ))|$, for some $s(\alpha_{p} ), r(\alpha_{p} ) \in [0,\alpha_p] $. From Corollaries \ref{distance} and \ref{k_j} , there exist constants $C_k$ and $C_d$ such that $|\dot{k}_j(s(\alpha_p))|  \leq C_k$ and $|\dot{d}_j(s(\alpha_{p} ))|  \leq C_d$. Therefore for all $j$, \newline $| k_j(\alpha_{p} ) - k_j(0)|  \leq \alpha_{p}  C_k < \delta C_k$, and $| d_j (\alpha_{p} ) - d_j(0)| \leq \alpha_{p}  C_d < \delta C_d $. Then



\begin{equation*}
\begin{aligned}
 \Big| \lambda_1 ^ m (\alpha_{p} ) - \lambda_1 ^ m (0) \Big |  &\leq \frac{1}{m} \, {C}_0 \sum_{j=1}^{m}  \Big(  \Big| {d_j (\alpha_{p} ) - d_j (0)\Big|} k_j (\alpha_{p} ) + d_j (0)  \Big| {k_j (\alpha_{p}  ) - k_j(0)}   \Big| \Big) \\
&<  \frac{1}{m} \, {C}_0 \sum_{j=1}^{m} \delta (C_d k_{\max} + C_k d_{\max} ) \\
%
%
&= {C}_0  \delta (C_d k_{\max} + C_k d_{\max} ) <  \varepsilon,  \\
\end{aligned}
\end{equation*}

\bigskip 

\noindent if we take $\delta < \frac{\varepsilon}{C_d k_{\max} + C_k d_{\max} }$. We now have a contradiction because with the choice of the sequence $ \alpha_1 > \alpha_2 > ... > \alpha_{p} > ... \to 0$ in $[0, b]$. Therefore the statement is proved. 
\end{proof}


\section{Differentiability of the largest Lyapunov exponent}

Here we prove {\bf {Theorem \ref{diff.}}} 

%

\begin{proof}[Proof of {\rm {Theorem \ref{diff.}:}}]
We will prove differentiability at $\alpha=0$. From this differentiability at any $\alpha \in [0, b]$ follows. To prove the differentiability at $\alpha =0$, we have to show that there exists 
\begin{equation*}
\begin{aligned}
\lim_{\alpha \to 0}\frac {\lambda_1 (\alpha) - \lambda_1 (0) } {\alpha}.
\end{aligned}
\end{equation*}

\noindent Equivalently, there exists a number $F$ such that 

\begin{equation*}
\begin{aligned}
\lim_{p \to \infty}\frac {\lambda_1 (\alpha_p) - \lambda_1 (0) } {\alpha_p} = F,
\end{aligned}
\end{equation*}

\noindent for any sequence $\alpha_1 > \alpha_2 > ... > \alpha_{p} > ... \to 0$ as $p \to \infty $ in $[0, b]$.

 \bigskip 
 
\noindent Let $K(\alpha) \subset \mathbb{R}^2$  be a $\mathcal{C}^{5, 3}$ billiard deformation and $\alpha \in [0,b]$ for a positive number $b$. Let $\lambda_1  (\alpha)$ be the largest Lyapunov exponent for $K(\alpha)$ and $\lambda_1  (0)$ be the largest Lyapunov exponent for $K(0)$. 
\bigskip

 \noindent By using Lemma \ref{alpha_tilde{k}} and the expressions of $\lambda_1^{m}(\alpha)$ for $\alpha = \alpha_p$ and $\lambda_1^{m}(0)$ in (\ref{lambda_k}), we have $ \lambda_1 ^ {(m)} (\alpha_p) \to \lambda_1  (\alpha_p)$ and $ \lambda_1 ^ {(m)} (0) \to \lambda_1  (0)$ when $m \to \infty$. Also, 

\begin{equation*}
\begin{aligned}
\frac{ \lambda_1 ^ {(m)} (\alpha_p) - \lambda_1 ^ {(m)} (0)}{\alpha_p}  &= - \frac{1}{m} \sum_{j=1}^{m}\frac{\log \delta_j (\alpha_p)- \log\delta_j (0)}{\alpha_p}  \\
 &= - \frac{1}{m} \sum_{j=1}^{m}\frac{\log \big(1+d_j (\alpha_p) k_j(\alpha_p)\big)- \log \big(1+d_j (0) k_j(0)\big)}{\alpha_p}.  \\
 \end{aligned}
\end{equation*}

\noindent Set $f_j (\alpha_p) = \log \big (1+d_j (\alpha_p) k_j(\alpha_p) \big)$ and $f_j (0) = \log \big (1+d_j (0) k_j(0) \big)$. Then

\begin{equation*}
\begin{aligned}
 \frac{ \lambda_1 ^ {(m)} (\alpha_p) - \lambda_1 ^ {(m)} (0)}{\alpha_p} &= - \frac{1}{m} \sum_{j=1}^{m}\frac{ f_j (\alpha_p)- f_j (0)}{\alpha_p}.  \\
\end{aligned}
\end{equation*}

\noindent Taylor's formula gives

\begin{equation*}
\begin{aligned}
f_j(\alpha_p) &= f_j(0) + \alpha_p \dot{f}_j (0) + \frac{\alpha_p^2}{2} \ddot{f}_j(r_j(\alpha_p))
\end{aligned}
\end{equation*}
 
 \noindent for some $r_j(\alpha_p) \in [0, \alpha_p]$. Then 
 \begin{equation*}
\begin{aligned}
\frac{f_j(\alpha_p) - f_j(0)}{\alpha_p} - \dot{f}_j (0) &= \frac{\alpha_p}{2} \ddot{f}_j(r_j(\alpha_p)).
\end{aligned}
\end{equation*}
 
\noindent Let \, $$ F_m = \frac{1}{m} \sum_{j=1}^{m} \dot{f}_j(0).$$ 

\noindent Summing up the above for $j=1,2,..., m$, we get

\begin{equation*}
\begin{aligned}
\frac{ \lambda_1 ^ {(m)} (\alpha_p) - \lambda_1 ^ {(m)} (0)}{\alpha_p}  - F_m &= - \frac{1}{m} \sum_{j=1}^{m} \Big[\frac{ f_j (\alpha_p)- f_j (0)}{\alpha_p} - \dot{f}_j (0) \Big]. \\
\end{aligned}
\end{equation*}

\noindent From the definition of $f_j(\alpha_p)$,

\begin{equation*}
\begin{aligned}
\dot{f}_j(\alpha_p) &= \frac{\dot{d}_j(\alpha_p) k_j(\alpha_p) + d_j(\alpha_p) \dot{k}_j(\alpha_p) } {1+d_j(\alpha_p) k_j(\alpha_p) }, \\
\end{aligned}
\end{equation*} 

\noindent and therefore, 

\begin{equation*}
\begin{aligned}
\ddot{f}_j(\alpha_p) & = \frac{ \big(\ddot{d}_j(\alpha_p) k_j(\alpha_p) + 2 \dot{d}_j(\alpha_p) \dot{k}_j(\alpha_p) +d_j(\alpha_p) \ddot{k}_j(\alpha_p)\big) \big( 1+d_j(\alpha_p) k_j(\alpha_p) \big) } {\big(1+d_j(\alpha_p) k_j(\alpha_p) \big)^2}\\
&  - \frac{\big(\dot{d}_j(\alpha_p) k(\alpha_p) + d_j(\alpha_p) \dot{k}_j(\alpha_p)\big)^2 } {\big(1+d_j(\alpha_p) k_j(\alpha_p) \big)^2}.
\end{aligned}
\end{equation*} 

\noindent Then from Corollaries \ref{distance} and \ref{k_j}, we get  

\begin{equation*}
\begin{aligned}
\Big |\dot{f}_j(\alpha_p) \Big| & \leq \frac{C_d^{(1)} k_{\max} + d_{\max} C_k^{(1)} } {1+d_{\min} k_{\min} } = C_1, \\
\\
\Big |\ddot{f}_j(\alpha_p) \Big| & \leq  \frac{ \big(C^{(2)}_d k_{\max} + 2 C_d^{(1)} C^{(1)} _k +d_{\max} C^{(2)}_k \big) \big( 1+d_{\max} k_{\max} \big) } {\big(1+d_{\min} k_{\min} \big)^2} \\
&+ \frac {\big(C^{(1)}_d k_{\max} + d_{\max} C^{(1)}_k \big)^2 } {\big(1+d_{\min} k_{\min} \big)^2} = C_2. 
\end{aligned}
\end{equation*} 

\noindent Therefore $$ | \ddot{f} _j(r_j(\alpha_p)) | \leq C_2 , $$

\noindent for some constant $C_2 > 0$ independent of $r_j(\alpha_p)$ and $j$. This implies

\begin{equation*}
\begin{aligned}
\Big| \frac{ \lambda_1 ^ {(m)} (\alpha_p) - \lambda_1 ^ {(m)} (0)}{\alpha_p}  - F_m \Big| & \leq  \frac{1}{m} \sum_{j=1}^{m}\frac{\alpha_p }{2} \Big| \ddot{f}_j (t_j(\alpha_p))\Big|  
&\leq \frac{C_2}{2} \alpha_p.
\end{aligned}
\end{equation*}


\bigskip

\noindent Since $|\dot{f}_j(\alpha_p)| \leq C_1$, we have $|F_m| \leq \frac{1}{m} \sum_{j=1}^{m} |\dot{f}_j(0)| \leq C_1$, for all $m$. Therefore, the sequence $\{F_m\}$ has convergent subsequences. Let for example $ F_{m_h} \to {F}$, for some sub-sequence $\{m_h\}$. Then 

\begin{equation*}
\begin{aligned}
\Big | \frac{ \lambda^{(m_h)}_1  (\alpha_p) - \lambda^{(m_h)}_1  (0)}{\alpha_p}  - F_{m_h} \Big | &\leq \frac{C_2}{2} \alpha_p, 
\end{aligned}
\end{equation*}

\noindent for all $h\geq 1$. So, letting $h \to \infty $, we get 

\begin{equation*}
\begin{aligned}
\Big | \frac{ \lambda_1  (\alpha_p) - \lambda_1  (0)}{\alpha_p}  - {F} \Big | &\leq \frac{C_2}{2} \alpha_p, 
\end{aligned}
\end{equation*}

\noindent and letting $ \alpha_p \to 0$ as $p \to \infty$ we get that there exists 
\begin{equation*}
\begin{aligned}
\lim_{p \to \infty} \frac{ \lambda_1  (\alpha_p) - \lambda_1  (0)}{\alpha_p}  = F.  
\end{aligned}
\end{equation*} 

\noindent for every sequence $\alpha_1 > \alpha_2 > ... > \alpha_{p} > ... \to 0$ as $p \to \infty $ in $[0, b]$. Thus, there exists $ F = \lim_{m \to \infty} \frac{1}{m} \sum_{j=1}^{m} \dot{f}_j(0)$. This is true for every subsequence $\{m_h\}$, so for any subsequence we have $ F_{m_h} \to F$. Hence, $ F_{m}$ converges to $F$ as well. This implies that there exists

\begin{equation*}
\begin{aligned}
\lim_{\alpha \to 0} \frac{ \lambda_1  (\alpha) - \lambda_1  (0)}{\alpha}  = F,  
\end{aligned}
\end{equation*}

\noindent so $\lambda_1$ is differentiable at $\alpha =0$ and $\dot{\lambda}_1(0) =F$. 
\end{proof} 

\begin{cor}
Let $K(\alpha)$ be a $\mathcal{C}^{5, 3}$ billiard deformation. Then there exists a constant 
$C _{\lambda_1} > 0$ such that 

$$ \Big | \diff{\lambda_1(\alpha)}{\alpha} \Big | \leq C _{\lambda_1}, \, \, \, \, for\, all\, \alpha \in [0,b].$$

\end{cor}

\begin{proof}
We have

\begin{equation*}
\begin{aligned}
\lambda_1 (\alpha) &= \lim_{m\to \infty} \frac{1}{m} \sum_{j=1}^{m} \log(1+d_j (\alpha) k_j (\alpha)).\\
\end{aligned}
\end{equation*}

\noindent By Theorem \ref{diff.}, $\lambda_1(\alpha)$ is $\mathcal{C}^1$. So, from the formula 
 in the previous proof that \newline
 $ \dot{\lambda}_1 (0) = \lim_{m \to \infty} \frac{1}{m} \sum_{j=1}^{m} \dot{f}_j(0)$, we have 

\begin{equation*}
\begin{aligned}
\diff{\lambda_1}{\alpha} &= \lim_{m\to \infty} \frac{1}{m} \sum_{j=1}^{m} \frac{\diff{d_j}{\alpha} k_j(\alpha) + d_j(\alpha) \diff{k_j}{\alpha}} {1+d_j(\alpha) k_j (\alpha)}. 
\end{aligned}
\end{equation*}

\noindent From Corollaries \ref{distance} and \ref{k_j}, there exist constants $C_d^{(1)}, C_k^{(1)} >0 $ such that \newline
$ \Big | \diff{d_j} {\alpha} \Big | \leq C_d^{(1)}$ and $ \Big | \diff{k_j} {\alpha} \Big | \leq C_k^{(1)} $. Then, we have

\begin{equation*}
\begin{aligned}
\Big |\diff{\lambda_1}{\alpha} \Big |& \leq  \lim_{m\to \infty} \frac{1}{m} \sum_{j=1}^{m} \frac{C_d^{(1)} k_{\max} + d_{\max} C_k^{(1)}} {1+d_{\min} k_{\min}}\\ 
&= \frac{C_d^{(1)} k_{\max} + d_{\max} C_k^{(1)}} {1+d_{\min} k_{\min}}\\
&= C _{\lambda_1}. 
\end{aligned}
\end{equation*}

\noindent This proves the statement. 
\end{proof}

\bigskip

\section*{Acknowledgment}
The author would like to thank Prof. Luchezar Stoyanov for his suggestions, comments, and help. This work was supported by a scholarship from Najran University, Saudi Arabia.  





\end{document}